\documentclass[12pt]{amsart}
\usepackage[vmargin=1in,hmargin=1.2in]{geometry}
\usepackage{graphicx, latexsym, graphics}
\usepackage{amsfonts, amssymb, amsmath, amsthm, mathtools, mathrsfs}
\usepackage{algorithmic, algorithm}
\usepackage{tikz}
\usetikzlibrary{matrix,arrows}

\newcommand{\ds}{\displaystyle}
\newcommand{\NN}{\mathbb N}

\newcommand{\CC}{\mathbb C}
\newcommand{\RR}{\mathbb R}
\newcommand{\ZZ}{\mathbb Z}

\newcommand{\DD}{\mathcal D}
\newcommand{\SSS}{\mathcal S}

\theoremstyle{plain}
\newtheorem{theorem}{Theorem}[section]
\newtheorem{proposition}[theorem]{Proposition}

\newtheorem{corollary}[theorem]{Corollary}

\theoremstyle{remark}
\newtheorem{remark}[theorem]{Remark}

\theoremstyle{definition}
\newtheorem{definition}[theorem]{Definition}

\allowdisplaybreaks

\numberwithin{equation}{section}

\newcommand{\beq}{\begin{eqnarray}}
\newcommand{\eeq}{\end{eqnarray}}

\newcommand{\beqs}{\begin{eqnarray*}}
\newcommand{\eeqs}{\end{eqnarray*}}

\newcommand{\supp}{\mathrm{supp\,}}

\begin{document}

\title[Translation-modulation invariant Banach spaces]
{Translation-modulation invariant Banach spaces of ultradistributions}

\author[P. Dimovski]{Pavel Dimovski}
\address{P. Dimovski, Faculty of Technology and Metallurgy, University Ss. Cyril and Methodius, Ruger Boskovic 16, 1000 Skopje, Macedonia}
\email{dimovski.pavel@gmail.com}

\author[S. Pilipovi\'{c}]{Stevan Pilipovi\'{c}}
\thanks{S. Pilipovi\'{c} gratefully acknowledges support by the Project 174024 of the Serbian Ministry for Education, Science and  Technological Development}
\address{S. Pilipovi\'{c}, Department of Mathematics and Informatics, University of Novi Sad, Trg Dositeja Obradovi\'ca 4, 21000 Novi Sad, Serbia}
\email {stevan.pilipovic@dmi.uns.ac.rs}

\author[B. Prangoski]{Bojan Prangoski}
\address{B. Prangoski, Faculty of Mechanical Engineering, University Ss. Cyril and Methodius, Karpos II bb, 1000 Skopje, Macedonia}
\email{bprangoski@yahoo.com}

\author[J. Vindas]{Jasson Vindas}
\thanks{J. Vindas  gratefully acknowledges support by Ghent University, through the BOF-grants 01J11615 and 01N01014.}
\address{J. Vindas, Department of Mathematics, Ghent University, Krijgslaan 281, 9000 Gent, Belgium}
\email{jasson.vindas@UGent.be}


\maketitle

\begin{abstract}
We introduce and study a new class of translation-modulation invariant Banach spaces of ultradistributions.
These spaces show stability under Fourier transform and tensor products; furthermore, they have a natural Banach convolution module structure over a certain associated Beurling algebra, as well as a Banach multiplication module structure over an associated Wiener-Beurling algebra. We also investigate a new class of modulation spaces, the Banach spaces of ultradistributions $\mathcal{M}^F$ on $\mathbb{R}^{d}$, associated to translation-modulation invariant Banach spaces of ultradistributions $F$  on $\mathbb{R}^{2d}$.
\end{abstract}
\maketitle

\section{Introduction}

 In this article we introduce and study a class of Banach spaces of ultradistributions that are invariant under translation and modulation operators. We shall also use these Banach spaces to define new classes of modulation spaces.

The modulation spaces were introduced by H. Feichtinger \cite{F0,F} and provide a natural framework to study many aspects of time-frequency analysis. Their properties were investigated in detail by him and K. Gr\"{o}chenig in \cite{F1, FG88, fg11,fg22, F90}. See the monograph \cite{gr} for an overview of results and applications. Modulation spaces have been considered by many authors both in the setting of distributions and ultradistributions \cite{C-P-R-T,FG,GZ1,PT,PilT,Toft2012}, and they have shown useful in the study of pseudo-differential and localization operators. The Banach spaces we introduce here allow us to consider natural generalizations of the $M^{p,q}_{\omega}(\RR^{d})$ spaces and further refine the scale of modulation spaces. Our new modulation spaces turn out to be larger or smaller than other extreme spaces usually considered in literature, but they still enjoy good translation-modulation invariance properties.

We define the class of so-called translation-modulation invariant Banach spaces of ultradistributions (TMIB) in Section \ref{TMI}. They are Banach spaces $E$ that satisfy the dense and continuous inclusions $\mathcal{S}^*_{\dagger}(\mathbb{R}^d)\hookrightarrow E\hookrightarrow \mathcal{S}'^*_{\dagger}(\RR^d)$ with respect to Gelfand-Shilov spaces, are invariant under the translations and modulations, and the operator norms of translation and modulation operators satisfy certain growth bounds with respect to the regularity of the Gelfand-Shilov space.  We discuss in this section stability properties of the TMIB of ultradistributions and their duals (DTMIB of ultradistributions) under topological tensor products; we also show that they have a natural Banach convolution module structure over a certain associated Beurling algebra, as well as a Banach multiplication module structure over an associated Wiener-Beurling algebra. We point out that the Banach space $E$ is not assumed to be a solid Banach space in the sense of \cite{f1979}; indeed, the elements of $E$ may not be $L^{1}_{loc}$ functions and actually $E$ needs not even contain non-trivial compactly supported functions. We also mention that our considerations apply to spaces of both quasianalytic and non-quasianalytic type. Therefore, in the quasianalytic case, our TMIB of ultradistributions are not necessarily so-called Banach spaces in `standard situation' in the sense of W. Braun and H. Feichtinger \cite{BF,F84} either, because their associated Wiener-Beurling algebras may consist of quasianalytic functions.

 We study in Section \ref{MS} a generalization of the modulation spaces, the space $\mathcal{M}^F$ consisting of ultradistributions on $\mathbb{R}^{d}$ whose short-time Fourier transforms belong to a given $F$ that is assumed to be either a TMIB or a DTMIB of ultradistributions on $\mathbb{R}^{2d}$. The main difference between our setting in this paper and that of \cite{fg11} is the absence of solidity assumptions on $F$ and of localization arguments in the quasianalytic case. Despite the general character of $F$, we show here that the Banach space $\mathcal{M}^{F}$ enjoys most of the useful properties of the standard modulation spaces $M^{p,q}_{\omega}(\mathbb{R}^{d})= \mathcal{M}^{L^{p,q}_{\omega}(\mathbb{R}^{2d})}$. Our results on (completed) tensor products from Section \ref{TMI} provide a way to generate a number of interesting instances of TMIB and DTMIB of ultradistributions on $\mathbb{R}^{2d}$; their associated modulation spaces, in general, differ from $M^{p,q}_{\omega}(\mathbb{R}^{d}).$ As an example, we discuss in Section \ref{ex} the case of $\mathcal{M}^{L^2(\RR^d)\hat{\otimes}_{\pi} L^2(\RR^d)}$, which is a proper subspace of $L^{2}(\mathbb{R}^{d})$; here $L^2(\RR^d)\hat{\otimes}_{\pi} L^2(\RR^d)$ stands for the completion of  $L^2(\RR^d)\otimes L^2(\RR^d)$ with respect to the $\pi$-topology \cite{ryan,Treves}.

\section{Notation and Preliminaries}\label{notation}
\label{notation}

Let $M_p$ and $A_p$, $p\in\NN$, be two weight sequences of positive numbers such that $M_0=M_1=A_0=A_1=1$. Throughout the article we assume that both sequences satisfy the following three conditions:

\smallskip

\indent $(M.1)$ $M_{p}^{2} \leq M_{p-1} M_{p+1}, \; \; p \in\ZZ_+$;\\
\indent $(M.2)$ $\ds M_{p} \leq c_0H^{p} \min_{0\leq q\leq p} \{M_{p-q} M_{q}\}$, $p,q\in \NN$, for some $c_0,H\geq1$;\\
\indent $(M.6)$ $p!\subset M_p$; i.e., there exist $c_0,L_0\geq1$
such that $p!\leq c_0 L_0^p M_p$, $p\in\NN$.

\smallskip

The meaning of these standard conditions in the theory of ultradifferentiable functions and ultradistributions is very well explained in \cite{Komatsu1}. Without any loss of generality, we can assume the constants $c_0$, $H$ and $L_0$ that appear in these conditions are the same for both sequences $M_p$ and $A_p$.
The special cases of the Gevrey sequences $M_p=p!^{\lambda}$ and $A_p=p!^{\tau}$, for $\lambda,\tau\geq 1$, are classical examples of weight sequences satisfying our assumptions. Observe also that $M_p$ and $A_p$ may or may not be quasianalytic \cite{Komatsu1}; although such a distinction plays basically no role in this paper, we refer to \cite{beurling1938,D,Ho1,Komatsu1,RS} for properties of quasianalytic and non-quasianalytic functions.

We denote as $M(\cdot)$ and $A(\cdot)$ the associated functions of $M_p$ and $A_p$, namely, $M(\rho)=\sup_{p\in\NN}\ln_+(\rho^p/M_p)$, for $\rho\geq 0$,  and $A(\rho)$ likewise defined. They are non-negative, continuous, non-decreasing functions that vanish for sufficiently small $\rho$ and grow more rapidly than $\ln \rho^n$ for any $n\in\ZZ_+$ as $\rho\to\infty$ (see \cite{Komatsu1}). For example, if $M_p=p!^{\lambda}$, one has $M(\rho)\asymp \rho^{1/\lambda}$. Given $h>0$, we often write $M_{h}(\rho)=M(h\rho)$ and $A_{h}(\rho)=A(h\rho)$.

Let $h>0$. We denote as $\SSS^{M_p,h}_{A_p,h}$ the Banach space of all smooth functions $\varphi \in C^{\infty}(\RR^d)$ for which the norm
\beqs
\sigma^{M_p,h}_{A_p,h}(\varphi)=\sup_{\alpha\in\NN^d} \frac{h^{|\alpha|}\left\|e^{A
_{h}}D^{\alpha}\varphi\right\|_{L^{\infty}(\RR^d)}} {M_{\alpha}}
\eeqs
is finite. One easily verifies that for $h_1<h_2$ the canonical inclusion $\SSS^{M_p,h_2}_{A_p,h_2}\rightarrow\SSS^{M_p,h_1}_{A_p,h_1}$ is compact.
As locally convex spaces (from now on abbreviated as l.c.s.) we define
$$\SSS^{(M_p)}_{(A_p)}(\RR^d)=\lim_{\substack{\longleftarrow\\ h\rightarrow \infty}} \SSS^{M_p,h}_{A_p,h} \quad\mbox{ and }\quad
\SSS^{\{M_p\}}_{\{A_p\}}(\RR^d)=\lim_{\substack{\longrightarrow\\ h\rightarrow 0^{+}}} \SSS^{M_p,h}_{A_p,h}.$$
The space $\SSS^{(M_p)}_{(A_p)}(\RR^d)$ is an $(FS)$-space (Fr\'echet-Schwartz space), while $\SSS^{\{M_p\}}_{\{A_p\}}(\RR^d)$ is a $(DFS)$-space (dual Fr\'echet-Schwartz space), in particular they are Montel spaces; for the abstract theory concerning $(FS)$ and $(DFS)$-spaces we refer to \cite[Appendix A]{morimoto}.

For $h>0$ and $k\in\NN$, we also define the Banach spaces
\begin{align*}
\mathcal{S}_{A_p,h}^{\emptyset,k}&=\{\varphi\in C^k(\RR^d)|\,\sigma_{A_p,h}^{\emptyset,k}(\varphi)<\infty\},\\
\mathcal{S}_{\emptyset,k}^{M_p,h}&=\{\varphi\in C^\infty(\RR^d)|\,\sigma_{\emptyset,k}^{M_p,h}(\varphi)<\infty\},\\
\mathcal{S}_{\emptyset,h}^{\emptyset,k}&=\{\varphi\in C^k(\RR^d)|\,\sigma_{\emptyset,h}^{\emptyset,k}(\varphi)<\infty\},\\
\end{align*}
where the corresponding norms are
\[
\sigma_{A_p,h}^{\emptyset,k}(\varphi)=\sup_{|\alpha|\leq k}\|e^{A_{h}} D^\alpha \varphi\|_{L^\infty(\RR^d)}, \quad
\sigma_{\emptyset,k}^{M_p,h}(\varphi)=\sup_{\alpha\in\mathbb{N}^d} \frac{h^{|\alpha|}\|(1+|\cdot|)^k D^\alpha\varphi\|_{L^\infty(\RR^d)}}{M_\alpha},
\]
and
$\sigma_{\emptyset,h}^{\emptyset,k}(\varphi)=\sup_{|\alpha|\leq k}\|(1+|\cdot|)^k D^\alpha\varphi\|_{L^\infty(\RR^d)}. $
Clearly, the Schwartz space is
$\ds\mathcal{S}(\mathbb{R}^d)=\lim_{\substack{\longleftarrow\\ k\rightarrow \infty}} \mathcal{S}_{\emptyset,k}^{\emptyset,k}$ and, similarly,  we define  the following $(FS)$-spaces of Beurling type,
\begin{equation*}
\mathcal{S}_{(A_p)}^\emptyset (\RR^d)=\lim_{\substack{\longleftarrow\\ k\rightarrow \infty}} \mathcal{S}_{A_p,k}^{\emptyset,k},  \quad
\mathcal{S}_{\emptyset}^{(M_p)}(\RR^d) = \lim_{\substack{\longleftarrow\\ k\rightarrow \infty}} \mathcal{S}_{\emptyset,k}^{M_p,k}.
\end{equation*}
Note that $\mathcal{S}_{(A_p)}^\emptyset (\RR^d)$ is always of non-quasianalytic nature, that is, it contains the Schwartz space $\mathcal{D}(\mathbb{R}^{d})$ as a dense subspace.

Subsequently, we will use the common notation $\SSS^*_{\dagger}(\RR^d)$ for the five spaces $\SSS^{(M_p)}_{(A_p)}(\RR^d)$, $\SSS^{\{M_p\}}_{\{A_p\}}(\RR^d)$, $\SSS^{(M_p)}_{\emptyset}(\RR^d)$, $\SSS^{\emptyset}_{(A_p)}(\RR^d)$ and $\SSS(\RR^d)=\SSS^{\emptyset}_{\emptyset}(\RR^d)$. We endow the dual space $\SSS'^*_{\dag}(\RR^d)$ with the strong dual topology.  We emphasize again that if $\ast=\emptyset$ or $\dagger=\emptyset$, only the Beurling case is considered. Furthermore, for the sake of convenience, we shall always follow the ensuing \emph{convention}. If $\ast=\emptyset$, then the function $M_h$, $h>0$, stands for $M_h(\rho)=h\ln(1+|\rho|)$, and we still also call $M(\rho)=\ln(1+|\rho|)$ the associated function of this case. This convention also applies to $\dagger=\emptyset$ and in such case we write $A_h(\rho)=h\ln(1+|\rho|)$. We also mention that some of the proofs below will only be given in the cases $\ast=(M_p),\{M_p\}$ and $\dagger=(A_p),\{A_p\}$; the treatments of the other three cases only require small adjustments and are therefore omitted.

It is very important to note that $\SSS^*_{\dagger}(\RR^d)$ (and hence its strong dual $\SSS'^*_{\dagger}(\RR^d)$, cf. \cite[p. Proposition 50.6, p. 523]{Treves}) is nuclear in all cases considered in this article.
In fact, this observation is fundamental because it yields the validity of kernel theorems via standard arguments (cf., e.g., \cite[Subsection 2.3]{ppv} and \cite[Proposition 2]{BojanP}). The nuclearity of $\SSS(\RR^d)$ is well known. That of $\SSS^{(M_p)}_{(A_p)}(\RR^d)$ and $\SSS^{\{M_p\}}_{\{A_p\}}(\RR^d)$ is shown in \cite[Proposition 2.10]{ppv}; the same method applied in the quoted paper gives the nuclearity of $\SSS^{\emptyset}_{(A_p)}(\RR^d)$, whence we also obtain that $\SSS^{(M_p)}_{\emptyset}(\RR^d)$ is nuclear because it is isomorphic (as l.c.s.) to $\SSS^{\emptyset}_{(M_p)}(\RR^d)$ via Fourier transform. For future reference, we collect these facts in the next proposition.

\begin{proposition}\label{ktlshv135}
The spaces $\SSS^*_{\dag}(\RR^d)$ and $\SSS'^*_{\dag}(\RR^d)$ are nuclear. Furthermore, the following canonical isomorphisms of l.c.s. hold:
\beqs
\SSS^*_{\dag}(\RR^{d_1+d_2})\cong\SSS^*_{\dag}(\RR^{d_1})\hat{\otimes} \SSS^*_{\dag}(\RR^{d_2})\cong \mathcal{L}_b(\SSS'^*_{\dag}(\RR^{d_1}),\SSS^*_{\dag}(\RR^{d_2})),\\
\SSS'^*_{\dag}(\RR^{d_1+d_2})\cong\SSS'^*_{\dag}(\RR^{d_1})\hat{\otimes} \SSS'^*_{\dag}(\RR^{d_2})\cong \mathcal{L}_b(\SSS^*_{\dag}(\RR^{d_1}),\SSS'^*_{\dag}(\RR^{d_2})).
\eeqs
\end{proposition}

The completed tensor products above are taken with respect to either the $\epsilon$- or the $\pi$-topology, which amounts to the same in view of nuclearity \cite{Treves}.

We shall also employ the following notation. We fix constants in the Fourier transform as $\hat{f}(\xi)=\mathcal{F}f(\xi)=\int_{\RR^d} e^{-2\pi ix\cdot \xi}f(x)dx$, $f\in L^1(\RR^d)$. The Fourier transform is a topological isomorphism between $\SSS^*_{\dag}(\RR^d)$ and $\SSS^{\dag}_*(\RR^d)$, and thus between the ultradistribution spaces $\SSS'^*_{\dag}(\RR^d)$ and $\SSS'^{\dag}_*(\RR^d)$ by duality. Denote as $T_xf=f(\:\cdot\:-x)$ and $M_\xi f=e^{ 2\pi i\langle\xi,\cdot\rangle} f(\cdot)$ the operators of translation and modulation, respectively. Obviously, they act continuously on $\SSS^*_{\dag}(\RR^d)$ and, by duality, on $\SSS'^*_{\dag}(\RR^d)$ as well. For $f\in \mathcal{S}'^{*}_\dag (\mathbb{R}^d)$ and $x,\xi\in\RR^d$, we have $T_x M_\xi f=e^{-2\pi ix\cdot \xi} M_\xi T_x f$, $\mathcal{F}T_x f=M_{-x}\mathcal{F}f$ and $\mathcal{F}M_\xi f=T_{\xi} \mathcal{F}f$. Given a weight function $\eta:\mathbb{R}^{d}\to (0,\infty)$ and $1\leq p\leq \infty$, we denote as $L^p_{\eta}$ the weighted $L^p$ space of measurable functions $g$ such that $\|g\|_{L^{p}_{\eta}}:=\|\eta g\|_{p}<\infty$.

Finally, given a Banach space of ultradistributions $X\subseteq \SSS'^*_{\dag}(\RR^d)$, we define its associated Fourier space as the Banach space $\mathcal{F}X=\mathcal{F}(X) \subseteq \SSS'^{\dag}_{*}(\RR^d)$ with norm $\|\hat{f}\|_{\mathcal{F}X}=\|f\|_{X}$.

\section{Translation-modulation invariant  Banach spaces of ultradistributions}\label{TMI}
 We introduce in this section a general class of Banach spaces of ultradistributions that are invariant under the translation and modulation operators. An arrow $X\hookrightarrow Y$ always means continuous and dense inclusion between two l.c.s. In the sequel, if a separate treatment is needed, we first state assertions for the Beurling cases of $\ast$ and $\dagger$, followed by the Roumieu one in parenthesis.

\begin{definition}\label{defmod}
A Banach space $E$ is said to be a translation-modulation invariant Banach space of ultradistributions (in short: TMIB) of class $*-\dagger$ if it satisfies the following three conditions:
\begin{itemize}
    \item[(a)] The continuous and dense inclusions $\mathcal{S}^*_{\dagger}(\mathbb{R}^d)\hookrightarrow E\hookrightarrow \mathcal{S}'^*_{\dagger}(\RR^d)$ hold.
    \item[(b)] $T_{x}(E)\subseteq E$ and $M_\xi (E)\subseteq E$ for all $x,\xi\in\mathbb{R}^d$.
\item [(c)] There exist $\tau,C>0$ (for every $\tau>0$ there exists $C_{\tau}>0$), such that\footnote{Applying the closed graph theorem, the conditions (a) and (b) readily yield that $T_x\in\mathcal{L}(E)$ and $M_\xi\in\mathcal{L}(E)$ for every $x,\xi\in\RR^d$, see the proof of \cite[Lemma 3.1]{TIBU}. This justifies the fact that we can take their operator norms in \eqref{omega}.}
    \begin{equation}\label{omega}
   \omega_{E}(x):= \|T_{x}\|_{\mathcal{L}(E)}\leq C e^{A_{\tau}(|x|)} \quad  \mbox{and} \quad \nu_{E}(\xi) := \|M_{-\xi}\|_{\mathcal{L}(E)}\leq C e^{M_{\tau}(|\xi|)}.
   \end{equation}
\end{itemize}
The functions $\omega_{E}:\mathbb{R}^{d}\to (0,\infty)$ and $\nu_{E}:\mathbb{R}^{d}\to (0,\infty)$ defined in \eqref{omega} are called the weight functions of the translation and modulation groups of $E$, respectively (in short: its weight functions).
\end{definition}

\smallskip

In the case $\ast=\emptyset$, we also call a Banach space $E$ fulfilling the conditions of Definition \ref{defmod} a translation-modulation invariant Banach space of distributions. Note that in this case the bound in \eqref{omega} for the weight function of the modulation group reads $\nu_E(\xi)\leq C(1+|\xi|)^{\tau}$, for some $\tau,C>0$. If $\dagger=\emptyset$, the same comment applies to the bound for $\omega_E$.

We mention that Definition \ref{defmod} is intrinsically related to the notion of translation invariant Banach spaces of ultradistributions (TIB) studied by the authors in \cite{DPPV,DPV,DPV2015b,TIBU}. In fact, one readily verifies that $E$ is a TMIB of class $*-\dagger$ if and only if $E$ is a TIB of class $*-\dagger$ and $\mathcal{F}E$ is a TIB of class $\dagger-*$ (the latter readily follows from the identity $\mathcal{F}M_\xi f=T_{\xi} \mathcal{F}f$ and the definition of the norm on the Banach space $\mathcal{F}E$). Consequently,  $\mathcal{F}E$ is a TMIB of class $\dagger-*$ if $E$ is a TMIB of class $*-\dagger$. Note also that
\[\omega_{\mathcal{F}E}(x)=\check{\nu}_{E}(x)=\nu_{E}(-x)\quad  \mbox{and} \quad \nu_{\mathcal{F}E}(\xi)=\omega_{E}(\xi).
\]
Definition \ref{defmod} is similar to but not the same as the notion of Banach spaces in standard situation, introduced by Braun and Feichtinger \cite[Definition 3.1, p. 186]{BF}, see Remark \ref{TMrk6} below for a comparison.

We can use the results from \cite[Section 3]{TIBU} to deduce a number of important properties of a TMIB $E$. We start by pointing out that $E$ must be separable, as follows from property (a) of Definition
\ref{defmod}, and, by using properties (a) and (b), one obtains that, for any fixed $g\in E$,
\begin{equation}
\label{TMeq3.2}
x\mapsto T_xg \quad \mbox{and}
\quad \xi\mapsto M_{\xi}g
\end{equation}
are continuous mappings from $\RR^d$ to $E$; in other words, the translation and modulation operators on $E$ form both $C_0$-groups.

It is also clear that the weight functions $\omega_{E}$ and $\nu_{E}$ of a TMIB $E$ are both measurable (cf. \cite[p. 149]{TIBU}),  $\omega_{E}(0)=\nu_{E}(0)=1$, and their logarithms are subadditive functions. Using these properties of the weight functions, we can induce two natural Banach module structures on $E$ with respect to convolution and multiplication. Denote as $L^1_{\omega_E}$ the Beurling algebra on $\RR^{d}$ associated to the weight function of the translation group of $E$. We also employ the associated Fourier space $\mathcal{F}L^{1}_{\nu_{E}}$ corresponding to the weight function of the modulation group of $E$. Note that the multiplication of elements of $\mathcal{F}L^{1}_{\nu_{E}}$ can be defined via Fourier transform and the convolution operation on $L^{1}_{\nu_{E}}$, so that $\mathcal{F}L^{1}_{\nu_{E}}$ becomes a Banach algebra under this multiplication. We refer to it as the Wiener-Beurling algebra \cite{beurling1938,LST} with weight function $\nu_{E}$ and use the notation $A_{\nu_{E}}=\mathcal{F}L^{1}_{\nu_{E}}$. Clearly, when $\nu_{E}\geq 1$, the Wiener-Beurling algebra $A_{\nu_{E}}$ consists of continuous functions and its multiplication operation coincides with ordinary pointwise multiplication. Our results from \cite{TIBU} yield that $E$ is simultaneously a Banach convolution module over the Beurling algebra $L^{1}_{\omega_{E}}$ and a Banach multiplication module over the Wiener-Beurling algebra $A_{\nu_{E}}$, as follows:

\begin{proposition}
\label{propdoublemodule} Let $E$ be a TMIB of class $*-\dagger$. The convolution $\ast: \mathcal{S}^*_{\dagger}(\mathbb{R}^d)\times \mathcal{S}^*_{\dagger}(\mathbb{R}^d)\to \mathcal{S}^*_{\dagger}(\mathbb{R}^d)$
and multiplication $\cdot: \mathcal{S}^*_{\dagger}(\mathbb{R}^d)\times \mathcal{S}^*_{\dagger}(\mathbb{R}^d)\to \mathcal{S}^*_{\dagger}(\mathbb{R}^d)$  (uniquely) extend as continuous bilinear mappings
$$
\ast: L^{1}_{\omega_{E}}\times E \to E \quad \mbox{and} \quad
\cdot:  A_{\nu_{E}} \times E\to E
$$
such that $E$ becomes a Banach module over $L^{1}_{\omega_{E}}$ and $A_{\nu_{E}}$ with respect to these operations, that is,
\begin{equation}\label{convmodule}
\|g\ast f\|_E\leq \|g\|_{L_{\omega_{E}}^1}\|f\|_E,\quad g\in L_{\omega_E}^1,\ \forall f\in E,
\end{equation}
and
\begin{equation}\label{mult}
\|h \cdot f\|_E\leq \|h\|_{A_{\nu_{E}}}\|f\|_E,\quad \forall h\in A_{\nu_{E}},\  \forall f\in E.
\end{equation}
Furthermore, the convolution of $f\in E$ and $g\in L^{1}_{\omega_E}$ can be represented as a Bochner integral of an $E$-valued function, that is,
\beq\label{convint1}
g*f=\int_{\RR^d}g(y)T_{y}f dy,
\eeq
while its multiplication with $h\in A_{\nu_{E}}$ is given by the Bochner integral
\beq\label{multint1}
h\cdot f=\int_{\RR^d}(\mathcal{F}^{-1}h)(\xi) M_{-\xi}f d\xi.
\eeq
\end{proposition}
\begin{proof}
If $f,g\in \SSS^{\ast}_{\dagger}(\RR^{d})$, the relations \eqref{convmodule} and \eqref{convint1} follow directly from \cite[Lemmas 3.7 and 3.8]{TIBU}. Applying the same results to the TIB  $\mathcal{F}E$ of class $\dagger-*$ and mapping back with Fourier inverse transform, we obtain  (\ref{mult}) and (\ref{multint1}) (Bochner integrals commute with continuous linear mappings). The rest follows from a density argument (using also the Lebesgue dominated convergence theorem for Bochner integrals) because $\SSS^{\ast}_{\dagger}(\RR^{d})$ is dense in the three spaces $E,$ $L^{1}_{\omega_{E}}$, and $A_{\nu_E}$ (see, e.g., \cite[Lemma 3.9]{TIBU} for the latter two cases).
\end{proof}

The integral formulas in Proposition \ref{propdoublemodule} are a powerful tool; the proof of the following corollary is an elegant showcase for this.

\begin{corollary}\label{vsktjr157}
Let $E$ be a TMIB of class $*-\dagger$ and
let $\varphi,\chi\in \SSS^*_{\dag}(\RR^d)$ with $\varphi(0)=1$ and $\int_{\RR^d}\chi(x)dx=1$, and, for each $n\in\ZZ_+$, set $\varphi_n(x)=\varphi(x/n)$ and $\chi_n(x)=n^d\chi(nx)$. Given any $f\in E$, we have
$$\chi_n*(\varphi_n\cdot f)\rightarrow f\quad  \mbox{in }E.$$
 Furthermore, for each $g\in E'$ the functions $\varphi_n\cdot(\chi_n*g)$ belong to $\SSS^*_{\dag}(\RR^d)$ and
$$\varphi_n\cdot(\chi_n*g)\rightarrow g\quad \mbox{weakly* in } E'.$$
\end{corollary}

\begin{proof} Clearly $\int_{\RR^d}\mathcal{F}^{-1}\varphi_n(\xi)d\xi=1$. For $f\in E$, (\ref{convint1}) and (\ref{multint1}) imply\\
\\
$\|f-\chi_n*(\varphi_n\cdot f)\|_E$
\beqs
&\leq& \int_{\RR^d}|\mathcal{F}^{-1}\varphi_n(\xi)|\|f-M_{-\xi}f\|_Ed\xi+ \int_{\RR^d}|\chi_n(y)|\|\varphi_n\cdot f-T_y(\varphi_n\cdot f)\|_E dy\\
&\leq& \int_{\RR^d}|\mathcal{F}^{-1}\varphi(\xi)|\|f-M_{-\xi/n}f\|_Ed\xi\\
&{}&+ \int_{\RR^{2d}}|\chi(y)||\mathcal{F}^{-1}\varphi(\xi)| \|M_{-\xi/n}f-T_{y/n}M_{-\xi/n}f\|_Edyd\xi.
\eeqs
Now, Definition \ref{defmod} (c) together with dominated convergence proves $\chi_n*(\varphi_n\cdot f)\rightarrow f$ in $E$ (note that the mappings \eqref{TMeq3.2} are continuous). For $g\in E'\subseteq \SSS'^*_{\dag}(\RR^d)$, it is trivial that $\varphi_n(\chi_n*g)\in\SSS^*_{\dag}(\RR^d)$ (this follows by direct verification, but, alternatively, it also follows from general results such as \cite[Theorems 4.19 and 4.20]{d-v-ind2018}). The last part then follows by duality and the first part of the corollary.
\end{proof}

We are now able to obtain a characterization of compact subsets of a TMIB  of class $*-\dagger$ in the spirit of \cite{F84}. Indeed,

\begin{proposition}\label{newf}
Let $G$ be a closed subset of a TMIB space $E$ of class $*-\dagger$. Then, $G$ is compact in $E$ if and only if the following conditions hold:
\begin{itemize}
\item[(i)] $\sup_{f\in G}\|f\|_E=C<\infty$;
\item[(ii)] $\forall\varepsilon>0$, $\exists \theta\in\mathcal{S}^*_\dagger(\RR^d)$ such that  $\|\theta*f-f\|_E<\varepsilon$, $\forall f\in G$;
\item[(iii)] $\forall\varepsilon>0$, $\exists \psi\in\mathcal{S}^*_\dagger(\RR^d)$ such that  $\|\psi f-f\|_E<\varepsilon$, $\forall f\in G$.
\end{itemize}
\end{proposition}

\begin{proof} The proof is similar to those of \cite[Theorem 2.1 and Theorem 2.2]{F84}. The lack of functions with compact support in $\SSS^*_{\dagger}(\RR^d)$ when $M_p$ is quasianalytic requires to slightly modify the argument; however, the integral formulas of Proposition \ref{propdoublemodule} and the topological structure of $\SSS^*_{\dagger}(\RR^d)$ and its dual give a way to overcome this difficulty.\\
\indent Assume $G$ satisfies (i)--(iii). Let $\varepsilon>0$ be fixed. Take $\theta\in\SSS^*_{\dagger}(\RR^d)$ such that $\|f-\theta*f\|_E\leq \varepsilon/2$, $\forall f\in G$, and for this $\theta$ pick $\psi\in\SSS^*_{\dagger}(\RR^d)$ such that $\|f-\psi f\|_E\leq \varepsilon/(2(1+|\theta\|_{L^1_{\omega_E}}))$. Then
\beq\label{vstklc157}
\|f-\theta*(\psi f)\|_E\leq \|f-\theta*f\|_E+\|\theta*(f-\psi f)\|_E\leq \varepsilon,\,\, \forall f\in G.
\eeq
Let $\{f_n\}_{n\in\ZZ_+}$ be a sequence in $G$. Since $G$ is bounded in $E$ it is bounded in $\SSS'^*_{\dagger}(\RR^d)$, hence relatively compact in $\SSS'^*_{\dagger}(\RR^d)$ (the latter space is Montel). Furthermore, the topology induced on $G$ by $\SSS'^*_{\dagger}(\RR^d)$ is metrizable; in the $\{M_p\}-\{A_p\}$ case this is trivial since $\SSS'^{\{M_p\}}_{\{A_p\}}(\RR^d)$ is Fr\'echet and in the $(M_p)-(A_p)$, $\emptyset-(A_p)$, $(M_p)-\emptyset$ and $\emptyset-\emptyset$ cases this follows from \cite[Theorem 1.7, p. 128]{Sch} and \cite[Proposition 36.9, p. 376]{Treves}. Thus, there exists a subsequence of $\{f_n\}_n$ which, by abuse of notation, we still denote by $\{f_n\}_n$ that converges to $f\in \SSS'^*_{\dagger}(\RR^d)$ in $\SSS'^*_{\dagger}(\RR^d)$. Let $\varepsilon>0$ be arbitrary but fixed. Pick $\psi,\theta\in\SSS^*_{\dagger}(\RR^d)$ such that (\ref{vstklc157}) holds with $\varepsilon/3$. We claim that the mapping $u\mapsto \theta*(\psi u)$, $\SSS'^*_{\dagger}(\RR^d)\rightarrow \SSS^*_{\dagger}(\RR^d)$, is continuous. Clearly, it is well-defined, namely, $\theta*(\psi u)\in\SSS^*_{\dagger}(\RR^d)$. It is readily seen that it is continuous if we regard it as a mapping from $\SSS'^*_{\dagger}(\RR^d)$ into itself, hence its graph is a closed subspace of $\SSS'^*_{\dagger}(\RR^d)\times\SSS^*_{\dagger}(\RR^d)$. The Pt\'{a}k closed graph theorem \cite[Theorem 8.5, p. 166]{Sch} therefore proves its continuity ($\SSS^*_{\dagger}(\RR^d)$ is a Pt\'{a}k space, see the examples in \cite[Section 4.8, p. 162]{Sch}). Thus, the sequence $\{\theta*(\psi f_n)\}_n$ is Cauchy in $\SSS^*_{\dagger}(\RR^d)$ and hence in $E$ as well. Hence, there exists $n_0\in\ZZ_+$ such that for all $n,m\geq n_0$, $n,m\in\ZZ_+$, we have $\|\theta*(\psi f_n)-\theta*(\psi f_m)\|_E\leq \varepsilon/3$. Combining these facts, we conclude $\|f_n-f_m\|_E\leq \varepsilon$, $\forall n,m\geq n_0$, i.e., $\{f_n\}_n$ is Cauchy sequence in $E$ and hence $f\in E$ and $f_n\rightarrow f$ in $E$, which proves that $G$ is compact. Assume now that $G$ is compact. Clearly, (i) is satisfied. By using the integral formulas in Proposition \ref{propdoublemodule}, in the same way as in the proof of Corollary \ref{vsktjr157}, one can prove that for each $f\in E$ we have $\chi_n*f\rightarrow f$ and $\varphi_n f\rightarrow f$ in $E$, where $\chi_n$ and $\varphi_n$ are as in Corollary \ref{vsktjr157}. Thus, the continuous mappings $f\mapsto \chi_n*f$ and $f\mapsto \varphi_n f$, $E\rightarrow E$, tend to $\mathrm{Id}$ when we equip $\mathcal{L}(E,E)$ with the topology of simple convergence. The Banach-Steinhaus theorem implies that the convergence holds for the topology of precompact convergence and hence $G$ satisfies (ii) and (iii).
\end{proof}

A Banach space is called a dual translation-modulation invariant Banach space of ultradistributions (in short: DTMIB) of class $*-\dagger$ if it is the strong dual of a TMIB of class $*-\dagger$. If a DTMIB $E$ is reflexive, it is actually a TMIB of class $*-\dagger$, as follows from \cite[Proposition 3.14, p. 156]{TIBU}; however, this is not the case in general without reflexivity (e.g., property (a) from Definition \ref{defmod} fails for $E=L^{\infty}$). On the other hand, we always have that $E\hookrightarrow \mathcal{S}'^*_{\dagger}(\RR^d)$, properties (b) and (c) from Definition \ref{defmod}
hold for $E$, and \eqref{TMeq3.2} are weakly* continuous mappings. Furthermore, we can also provide $E$ with a convolution and multiplication Banach module structures. Define its weight functions also as in \eqref{omega}. Suppose $E=F'$ where $F$ is a TMIB of class $*-\dagger$. It is then clear that $\omega_{E}=\check{\omega}_{F}$ and $\nu_{E}=\nu_{F}$. For a given $f\in E$ we can then introduce its multiplication with $h\in A_{\nu_{E}}$ by duality: $\langle h f, \varphi\rangle = \langle f, h\varphi\rangle$, for $\varphi\in F$. Similarly, one defines the convolution $g\ast f$ for $g\in L^{1}_{\omega_{E}}$ (see \cite[Eq. (3.11), p. 156]{TIBU}). From Proposition (\ref{propdoublemodule}) one easily obtains the following dual version. (Hereafter for DTMIB, Pettis integrals are taken in the weak* topology of $E$ with respect to the duality $E=F'$.)

\begin{corollary}\label{TMIcordualmodule} Let $E$ be a DTMIB of class $*-\dagger$ . Then, \eqref{convmodule} and \eqref{mult} hold. Furthermore, the representations \eqref{convint1} and \eqref{multint1} are valid if interpreted as $E$-valued Pettis integrals.
\end{corollary}

We are also interested in the stability of TMIB and their duals under tensor products. So, the rest of this section discusses various methods how to construct TMIB on $\RR^{d_1+d_2}$ out of two TMIB on $\RR^{d_1}$ and $\RR^{d_2}$ as completed tensor products. Our next result makes use of the so-called approximation property, we refer to \cite[\S 43]{kothe2} and \cite[Section III.9]{Sch}  for this concept.

\begin{theorem}\label{EE'}
Let $E$ and $F$ be TMIB of class $*-\dag$ on $\mathbb{R}^{d_1}$ and $\RR^{d_2}$. Then,
\begin{itemize}
\item [(i)] $E\hat{\otimes}_\epsilon F$ is a TMIB of class $*-\dag$ on $\mathbb{R}^{d_1+d_2}$  and $\omega_{E\hat{\otimes}_\epsilon F}=\omega_{E}\otimes \omega_{F}$ and $\nu_{E\hat{\otimes}_\epsilon F}=\nu_{E}\otimes \nu_{F}$.
\item [(ii)] If either $E$ or $F$ satisfies the approximation property, then $E\hat{\otimes}_\pi F$ is also a TMBI class $*-\dag$ with $\omega_{E\hat{\otimes}_\pi F}=\omega_{E\hat{\otimes}_\epsilon F}$ and $\nu_{E\hat{\otimes}_\pi F}=\nu_{E\hat{\otimes}_\epsilon F}$.
\end{itemize}

\end{theorem}

\begin{proof}
First notice that $E\hat{\otimes}_{\pi} F$ and $E\hat{\otimes}_{\epsilon} F$ are Banach spaces (see \cite[Theorem 7, p. 178; Section 44.2, p. 267]{kothe2}). As $\SSS^*_{\dag}(\RR^n)$ and $\SSS'^*_{\dag}(\RR^n)$ are nuclear (cf. Proposition \ref{ktlshv135}), we have
\beqs
\SSS^*_{\dag}(\RR^{d_1})\otimes_{\pi=\epsilon}\SSS^*_{\dag}(\RR^{d_2})\xhookrightarrow{\mathrm{Id} \otimes_{\pi} \mathrm{Id}} E\otimes_{\pi} F\xhookrightarrow{} E\otimes_{\epsilon} F\xhookrightarrow{\mathrm{Id}\otimes_{\epsilon} \mathrm{Id}} \SSS'^*_{\dag}(\RR^{d_1})\otimes_{\pi=\epsilon}\SSS'^*_{\dag}(\RR^{d_2})
\eeqs
(cf. \cite[Theorem 8, p. 175; Section 41.5, p. 187; Theorem 6, p. 188; Theorem 1, p. 275; Theorem 1, p. 280]{kothe2}). For the moment, we denote as $\mathcal{B}_{\epsilon}(X'_{\sigma},Y'_{\sigma})$ the l.c.s. of separately continuous bilinear forms on $X'_{\sigma}\times Y'_{\sigma}$ equipped with the $\epsilon$-topology; it is complete when $X$ and $Y$ are complete (see \cite[Theorem 5, p. 167]{kothe2}). The diagram

\begin{tikzpicture}
  \matrix (m) [matrix of math nodes,row sep={3em}]
  {
   \SSS^*_{\dag}(\RR^{d_1+d_2}) & & & & \SSS'^*_{\dag}(\RR^{d_1+d_2}) \\
   \SSS^*_{\dag}(\RR^{d_1})\hat{\otimes} \SSS^*_{\dag}(\RR^{d_2}) & & & & \SSS'^*_{\dag}(\RR^{d_1})\hat{\otimes} \SSS'^*_{\dag}(\RR^{d_2})\\
   & E\hat{\otimes}_{\pi} F & & E\hat{\otimes}_{\epsilon} F & \\
   & & \mathcal{B}_{\epsilon}(E'_{\sigma},F'_{\sigma}) & &\\};
   \draw[right hook->] (m-1-1) -- (m-1-5) node[midway,above]{canonical inclusion};
   \draw[->] (m-2-1) -- (m-1-1) node[midway,left]{$\cong$};
   \draw[right hook->] (m-2-1) -- (m-3-2) node[midway,below]{$\mathrm{Id}\hat{\otimes}_{\pi}\mathrm{Id}\quad\quad$};
   \draw[right hook->] (m-2-1) -- (m-3-4) node[midway,above]{${}\quad\mathrm{Id}\hat{\otimes}_{\epsilon}\mathrm{Id}$};
   \draw[->] (m-3-2) -- (m-3-4) node[midway,below]{};
   \draw[->] (m-3-2) -- (m-4-3) node[midway,below]{};
   \draw[->] (m-3-4) -- (m-4-3) node[near start,below]{${}\quad$ can.} node[midway,below]{${}\quad\quad\quad$ inclusion};
   \draw[right hook->] (m-3-4) -- (m-2-5) node[midway,below]{${}\quad\quad\mathrm{Id}\hat{\otimes}_{\epsilon}\mathrm{Id}$};
   \draw[->] (m-2-5) -- (m-1-5) node[midway,right]{$\cong$};
\end{tikzpicture}

\noindent commutes, where, as usual, $\hookrightarrow$ denotes continuous and dense inclusion (see Proposition \ref{ktlshv135} and \cite[Theorem 5, p. 277]{kothe2}). Hence, $E\hat{\otimes}_{\epsilon} F$ satisfies Definition \ref{defmod} (a). If $E$ or $F$ satisfy the approximation property then \cite[Theorem 12, p. 240]{kothe2} implies that the canonical mapping $E\hat{\otimes}_{\pi} F\rightarrow \mathcal{B}_{\epsilon}(E'_{\sigma},F'_{\sigma})$ is injective. Now, the commutative diagram implies that we have the continuous injections 
$$
\SSS^*_{\dag}(\RR^{d_1+d_2}) \hookrightarrow E\hat{\otimes}_{\pi} F \to E\hat{\otimes}_{\epsilon} F \hookrightarrow \SSS'^*_{\dag}(\RR^{d_1+d_2}),  $$   so that $E\hat{\otimes}_{\pi} F$ satisfies Definition \ref{defmod} (a). Since $T_{(x,\xi)}=T_{x}\otimes T_{\xi}$ and $M_{(x,\xi)}=M_{x}\otimes M_{\xi}$ for all $(x,\xi)\in\RR^{2d}$, the space $E\hat{\otimes}_{\epsilon} F$ satisfies Definition \ref{defmod} (b) as well; analogously for $E\hat{\otimes}_{\pi} F$. Moreover, $T_{(x,\xi)}$ and $M_{(x,\xi)}$ are continuous on $E\hat{\otimes}_{\epsilon}F$ and $E\hat{\otimes}_{\pi} F$ and, obviously,
\beqs
\|T_{(x,\xi)}\|_{\mathcal{L}_b(E\hat{\otimes}_{\tau}F)}= \omega_E(x)\omega_F(\xi) \quad \mbox{and} \quad \|M_{(x,\xi)}\|_{\mathcal{L}_b(E\hat{\otimes}_{\tau}F)}= \nu_E(-x)\nu_F(-\xi)
\eeqs
with $\tau$ being either the $\pi$ topology or the $\epsilon$ topology (see \cite[Section 41.5, p. 187; Theorem 1, p. 275]{kothe2}). The proof of the theorem is complete.
\end{proof}

We end this section with some further remarks about TMBI and DTMBI of ultradistributions.

\begin{remark}
\label{TMrk1} Let $E$ and $F$ be as in Theorem \ref{EE'}.
\begin{itemize}
\item [(i)] The Banach space of integral operators $\mathcal{J}(E,F')$ (in the sense of Grothendieck) from $E$ to $F'$ equipped with the integral norm is a DTMIB of class $*-\dag$ on $\mathbb{R}^{d_{1}+d_{2}}$, because $\mathcal{J}(E,F')$ is isomorphic to the strong dual of $E \hat{\otimes}_{\epsilon}F.$
\item [(ii)] The dual of $(E \hat{\otimes}_{\pi} F)'$ is norm isomorphic to the space of linear bounded operators from $E$ into $F'$. Thus, if either $E$ or $F$ satisfies the approximation property, then $\mathcal{L}_b (E, F')$ is a DTMIB of class $*-\dag$ on $\mathbb{R}^{d_{1}+d_{2}}$.
\end{itemize}
\end{remark}

\begin{remark}\label{TMk2}
Let $\eta$ be a weight function of class $\dag$ on $\RR^{d}$, that is, a positive measurable function that satisfies the following condition: there exist $C,h>0$ (resp. for every $h>0$ there exists $C>0$) such that
\beq\label{trspkv157}
\eta(x+\xi)\leq C\eta(x)e^{A_{h}(|\xi|)},
\eeq
$\forall x,\xi\in\RR^{d}.$ The condition \eqref{trspkv157} ensures that $L^{p}_{\eta}$ is a TMIB of class $*-\dag$ on $\mathbb{R}^{d}$ if $1\leq p<\infty$, while $L^{\infty}_{\eta}$ is a  DTMIB of class $*-\dag$. Part (ii) of Theorem \ref{EE'} concerning $E\hat{\otimes}_{\pi} F$ is applicable when $E$ or $F$ are the spaces $L^p_{\eta}(\RR^d)$, $1\leq p<\infty$, or, if $\eta$ is in addition continuous and one of the spaces is $C_{\eta}(\RR^d)=\{\varphi\in C(\RR^d)|\, \varphi/\eta\in C_0(\RR^d)\}$, as all these spaces satisfy the approximation property.
\end{remark}

\begin{remark}
\label{TMrk3}
Let $E$ be a TMIB of class $*-\dag$ on $\mathbb{R}^{d_{1}}$ and let $\eta$ be a weight function of class $\dag$ on $\RR^{d_2}$. Note that the weighted Bochner-Lebesgue space $L^{p}_{\eta}(\mathbb{R}^{d_{2}},E)$ is a tensor product of $L^{p}_{\eta}$ and $E$ and can canonically be identified with a subspace of $\mathcal{S}'^*_{\dagger}(\RR^{d_1+d_2})$. If $1\leq p<\infty$, then
$L^{p}_{\eta}(\mathbb{R}^{d_{2}},E)$ is a TMIB of class $*-\dag$ on $\mathbb{R}^{d_1+d_2}$. The space $L^{p}_{\eta}(\mathbb{R}^{d_{2}},E')$ is a DTMIB of class $*-\dag$ if $1<p\leq \infty$ whenever $E'$ has the Radon-Nikodym property (in particular if $E$ is reflexive). Note that the mixed weighted Lebesgue spaces $L^{p,q}_{\eta_1\otimes\eta_{2}}(\RR^{d_1+d_2})$ arise in this way; in fact $L^{p,q}_{\eta_1\otimes\eta_{2}}(\RR^{d_1+d_2})= L^{q}_{\eta_2}(\mathbb{R}^{d_2},  L^{p}_{\eta_1}(\RR^{d_1}))$.
\end{remark}

\begin{remark}
In general, the completed tensor product of two $L^p$ spaces with respect to the $\pi$ topology does not have to be a solid Banach space (see \cite{F84} for the definition of a solid Banach space) even though its elements are all locally integrable functions. For example, when $1<p\leq 2$, $L^p(\RR^d)\hat{\otimes}_{\pi} L^p(\RR^d)$ is not a solid Banach space. This immediately follows from the following claim.

\smallskip

\noindent\textbf{Claim.} Let $1<p\leq 2$. There exists $F\in L^p(\RR^d)\hat{\otimes}_{\pi}L^p(\RR^d)$ such that $e^{2\pi i x\cdot \xi}F\not\in L^p(\RR^d)\hat{\otimes}_{\pi}L^p(\RR^d)$.

\smallskip

\begin{proof}
 Denote by $S$ the operator $f(x,\xi)\mapsto e^{2\pi i x\cdot \xi} f(x,\xi)$. Clearly it is a continuous operator on $\SSS(\RR^{2d})$ and on $\SSS'(\RR^{2d})$. Assume that for each $F\in L^p(\RR^d)\hat{\otimes}_{\pi}L^p(\RR^d)$ we have $SF\in L^p(\RR^d)\hat{\otimes}_{\pi}L^p(\RR^d)$. Observe that the mapping $F\mapsto SF$, $L^p(\RR^d)\hat{\otimes}_{\pi} L^p(\RR^d)\rightarrow \SSS'(\RR^{2d})$, is continuous since it decomposes as $\ds L^p(\RR^d)\hat{\otimes}_{\pi} L^p(\RR^d)\xrightarrow{\mathrm{Id}}\SSS'(\RR^{2d})\xrightarrow{S} \SSS'(\RR^{2d})$. Hence, the closed graph theorem implies that
\beq\label{113st}
S:L^p(\RR^d)\hat{\otimes}_{\pi} L^p(\RR^d)\rightarrow L^p(\RR^d)\hat{\otimes}_{\pi} L^p(\RR^d)\,\,\, \mbox{is continuous.}
\eeq
Consider the following bilinear functional on $L^p(\RR^d)\times L^p(\RR^d)$:
\beqs
B(f,g)=\int_{\RR^d}f(x)\mathcal{F}g(x)dx,\,\, L^p(\RR^d)\times L^p(\RR^d)\rightarrow \CC.
\eeqs
By applying the Hausdorff-Young inequality, one easily deduces that $B$ is continuous. Hence, $B$ naturally induces a continuous functional $B_1$ on $L^p(\RR^d)\hat{\otimes}_{\pi} L^p(\RR^d)$. As $\SSS(\RR^{2d})$ is canonically injected into $L^p(\RR^d)\hat{\otimes}_{\pi} L^p(\RR^d)$, its restriction $B_2$ is a continuous functional on $\SSS(\RR^{2d})$. Since the mapping
\beqs
\chi\mapsto \int_{\RR^{2d}}e^{-2\pi i x\cdot \xi}\chi(x,\xi)dxd\xi,\,\, \SSS(\RR^{2d})\rightarrow \CC,
\eeqs
is continuous and it coincides on $\SSS(\RR^d)\otimes\SSS(\RR^d)$ with $B_2$ it follows that $B_2$ is given by the above integral.\\
\indent Let $f(x)=f_1(x_1)f_2(x')$ (we denote $x=(x_1,x')$) where $0\neq f_2\in\DD(\RR^{d-1})$ is nonnegative and $f_1\in C^{\infty}(\RR)$ is given by $t\mapsto 1/t$ on $[2,\infty)$ and extended to be smooth and nonnegative with support in $[1,\infty)$. Let $\theta\in\DD(\RR)$ be such that $0\leq \theta\leq 1$, $\supp\theta\in[-2,2]$ and $\theta=1$ on $[-1,1]$. Define $\theta_n(t)=\theta(t/2^n)$ and denote $\varphi_n=(\theta_nf_1)\otimes f_2\in\DD(\RR^d)$, $n\in\ZZ_+$. Furthermore, pick a nonnegative $0\neq\psi\in\DD(\RR^d)$. Clearly $f\otimes \psi\in L^p(\RR^d)\hat{\otimes}_{\pi} L^p(\RR^d)$. We have
\beqs
B_1(S(\varphi_n\otimes\psi))&=&B_2(S(\varphi_n\otimes\psi))= \int_{\RR^{2d}}\varphi_n(x)\psi(\xi)dxd\xi\\
&=&\|\psi\|_{L^1(\RR^d)}\|f_2\|_{L^1(\RR^{d-1})} \int_{1}^{\infty}\theta_n(t)f_1(t)dt.
\eeqs
Since $\varphi_n\rightarrow f$ in $L^p(\RR^d)$, (\ref{113st}) implies that the left hand side tends to $B_1(S(f\otimes\psi))$. On the other hand, by monotone convergence, the right hand side tends to $\infty$.
\end{proof}
\end{remark}

\begin{remark}\label{TMrk4}
 In some cases, when $E$ and $F$ are classical spaces in analysis, $E\hat{\otimes}_{\pi} F$ or $E\hat{\otimes}_{\epsilon} F$ could happen to be a space of the same type. For example, if $\eta_1$ and $\eta_2$ are weight functions of class $\dag$ on $\RR^{d_1}$ and $\RR^{d_2}$ and $E=L^1_{\eta_1}(\RR^{d_1})$ and $F=L^1_{\eta_2}(\RR^{d_2})$, then $E\hat{\otimes}_{\pi} F$ is nothing else but $L^1_{\eta_1\otimes\eta_2}(\RR^{d_1+d_2})$. More generally \cite[Theorem 46.2, p. 473]{Treves}, we always have $L^{1}_{\eta_1}(\mathbb{R}^{d_1})\hat{\otimes}_{\pi} X= L^{1}_{\eta_1}(\mathbb{R}^{d_1},X)$ for $X$ an arbitrary Banach space. Similarly, if, in addition, $\eta_1$ and $\eta_2$ are continuous and $E=C_{\eta_1}(\RR^{d_1})$ and $F=C_{\eta_2}(\RR^{d_2})$, then $E\hat{\otimes}_{\epsilon} F$ is $C_{\eta_1\otimes\eta_2}(\RR^{d_1+d_2})$. On the other hand, it is important to notice that often this is not the case. To give an example, we recall the following result (see \cite[Theorem 4.21, p. 85]{ryan}). If $E$ and $F$ are reflexive Banach spaces such that either $E$ or $F$ satisfy the approximation property then the following statements are equivalent:
\begin{itemize}
\item[(i)] $E\hat{\otimes}_{\pi} F$ is reflexive;
\item[(ii)] the strong dual of $E\hat{\otimes}_{\pi} F$ is topologically isomorphic to $E'\hat{\otimes}_{\epsilon} F'$;
\item[(iii)] every $S\in\mathcal{L}(E,F')$ is compact;
\item[(iv)] $E'\hat{\otimes}_{\epsilon}F'$ is reflexive.
\end{itemize}
If we take $E=F=L^2(\RR^{d})$, then neither $L^2(\RR^d)\hat{\otimes}_{\epsilon} L^2(\RR^d)$ nor $L^2(\RR^d)\hat{\otimes}_{\pi} L^2(\RR^d)$ is isomorphic to $L^2(\RR^{2d})$, since, on the contrary, they would be reflexive and hence every continuous operator from $L^2(\RR^d)$ to $L^2(\RR^d)$ would have to be compact, which is obviously not true.

In fact, $L^2(\RR^d)\hat{\otimes}_{\pi}L^2(\RR^d)$ and $L^2(\RR^d)\hat{\otimes}_{\epsilon} L^2(\RR^d)$ are different from all $L^{p,q}(\RR^{2d})$, $1\leq p,q\leq\infty$. This should be clear for $1<p,q<\infty$ since, as seen above,  $L^2(\RR^d)\hat{\otimes}_{\pi}L^2(\RR^d)$ and $L^2(\RR^d)\hat{\otimes}_{\epsilon} L^2(\RR^d)$ are not reflexive. To see that they differ from $L^{p,1}(\RR^{2d})$, $1\leq p\leq\infty$, just take a function $f\in L^2(\RR^d)\backslash L^1(\RR^d)$ and any $\chi\in\DD(\RR^d)$. Then clearly $\chi\otimes f\in L^2(\RR^d)\otimes L^2(\RR^d)\subset L^2(\RR^d)\hat{\otimes}_{\pi} L^2 (\RR^d)\cap (L^2(\RR^d)\hat{\otimes}_{\epsilon} L^2(\RR^d))$, but $\chi\otimes f\not\in L^{p,1}(\RR^{2d})$. Similarly, for $L^{1,p}(\RR^{2d})$ and for $L^{\infty}(\RR^{2d})$.
\end{remark}

\begin{remark}\label{TMrk5}
If $M_p/p!$ and $A_{p}/p!$ are log-convex (i.e., they satisfy condition $(M.1)$), then there are TMIB of class $*-\dagger$ that are arbitrarily close to $\mathcal{S}^{*}_{\dagger}(\RR^{d})$ and $\mathcal{S}'^{*}_{\dagger}(\RR^{d})$ in the sense of  \eqref{TMeq3.8}, \eqref{TMeq3.10}, and \eqref{TMeq3.11} below. First of all, for $\lambda>0$, define
$$
\omega_{\lambda}(x)= \sum_{p=0}^{\infty} \frac{(\lambda |x|)^{p}}{M_p} \quad \mbox{and} \quad \nu_{\lambda}(x)= \sum_{p=0}^{\infty} \frac{(\lambda |x|)^{p}}{A_p}.
$$
It is well known \cite{Petzsche84} that $\ln \omega_{\lambda}$ and $\ln \nu_{\lambda}$  are subadditive under the assumption that $M_p/p!$ and $A_{p}/p!$ are log-convex. Moreover, using that $e^{M(\lambda |x|)}\leq \omega_{\lambda}(x)
\leq 2 e^{M(2\lambda |x|)}$ and  $e^{A(\lambda |x|)}\leq \nu_{\lambda}(x)
\leq 2e^{A(2\lambda |x|)}$, and the Chung-Chung-Kim characterization of the Gelfand-Shilov spaces \cite{Chung-C-K96}, we obtain

\begin{equation}
\label{TMeq3.8}
\SSS^{(M_p)}_{(A_p)}(\RR^d)=\lim_{\substack{\longleftarrow\\ \lambda \rightarrow \infty}}X_{\lambda}\quad\mbox{ and }\quad \SSS'^{(M_p)}_{(A_p)}(\RR^d)=\lim_{\substack{\longrightarrow\\ \lambda \rightarrow \infty}} X'_{\lambda}
\end{equation}
and
\begin{equation}
\label{TMeq3.9}
\SSS^{\{M_p\}}_{\{A_p\}}(\RR^d)=\lim_{\substack{\longrightarrow\\ \lambda \rightarrow 0}} X_{\lambda } \quad \mbox{and} \quad \SSS'^{\{M_p\}}_{\{A_p\}}(\RR^d)=\lim_{\substack{\longleftarrow\\ \lambda \rightarrow0}}X'_{\lambda}
\end{equation}
where the Banach spaces $X_{\lambda}$ and their strong dual $X'_{\lambda}$ are reflexive TMIB of class $(M_p)-(A_p)$ given  by
$$
X_{\lambda}= L^{2}_{\omega_{\lambda}}\cap \mathcal{F}L^{2}_{\nu_{\lambda}}, \quad \|\varphi\|_{X_{\lambda}}= \|\varphi\|_{L^{2}_{\omega_{\lambda}}}+ \|\mathcal{F}{\varphi}\|_{L^{2}_{\nu_{\lambda}}}.
$$

To obtain an analog of (\ref{TMeq3.8}) for the Roumieu case, namely, a representation of $\SSS^{\{M_p\}}_{\{A_p\}}(\RR^d)$ and $\SSS'^{\{M_p\}}_{\{A_p\}}(\RR^d)$ as the intersection and union of TMIB respectively, we consider the set $\mathfrak{R}$ of all unbounded non-decreasing sequences of positive real numbers. Given $(\lambda_j)\in\mathfrak{R}$, we set
$$
\omega_{(\lambda_j)}(x)= \sum_{p=0}^{\infty} \frac{|x|^{p}}{M_p\prod_{j=1}^{p}\lambda_{j}} \quad \mbox{and} \quad \nu_{(\lambda_{j})}(x)= \sum_{p=0}^{\infty} \frac{ |x|^{p}}{A_p\prod_{j=1}^{p}\lambda_{j}}
$$
and consider the Banach space $X_{(\lambda_j)}=L^{2}_{\omega_{(\lambda_j)}}\cap \mathcal{F}L^{2}_{\nu_{(\lambda_j)}}$. One also has that $\ln \omega_{(\lambda_j)}$ and $\ln \nu_{(\lambda_j)}$ are subadditive, so that $X_{(\lambda_j)}$ and $X'_{(\lambda_j)}$ are
reflexive TMIB of class $\{M_p\}-\{A_p\}$. Using \eqref{TMeq3.9} and reasoning as in \cite{Pilipovic}, we obtain, as l.c.s.,

\begin{equation}
\label{TMeq3.10}
\SSS^{\{M_p\}}_{\{A_p\}}(\RR^d)=\lim_{\substack{\longleftarrow\\(\lambda_j)\in\mathfrak{R}}}X_{(\lambda_j)},
\end{equation}
whence
\begin{equation}
\label{TMeq3.11}
\SSS'^{\{M_p\}}_{\{A_p\}}(\RR^d)=\bigcup_{(\lambda_j)\in\mathfrak{R}} X'_{(\lambda_j)}.
\end{equation}
\end{remark}

\begin{remark}
\label{TMrk6} Let $E$ be either a TMIB or a DTMIB of class $\ast-\dagger$. If $\nu_{E}$ is quasianalytic, that is, if
\begin{equation}
\label{weightquasieq}
\int_{\mathbb{R}^{d}} \frac{|\log \nu_{E}(\xi)|}{(1+|\xi|)^{d+1}} d\xi=\infty,
\end{equation}
then $E$ is not necessarily a Banach space in standard situation in the sense of \cite[Definition 1]{BF}. In fact, a necessary condition for a Banach space to be in standard situation (with respect to some `nice' Banach algebra) is to contain non-trivial compactly supported functions. The Banach space $E=L^{2}_{\nu}\cap \mathcal{F}L^{2}_{\nu}$ provides an instance of both a TMIB and a DTMIB that are not in standard situation. We also point out that in general the Wiener-Beurling algebra $A_{\nu_{E}}$ never satisfies the assumptions \cite[(A2) and (A4), p. 180]{BF} if (\refeq{weightquasieq}) holds, in view of Beurling's characterization of regularity  \cite{beurling1938}.

On the other hand, if $M_p$ is non-quasianalytic or $M_{p}=\emptyset$, one can show that $E$ is always in standard situation with respect some `nice' Banach algebra satisfying the assumptions $(A1)-(A5)$ from \cite{BF}. We leave the verification of this fact to the reader.
\end{remark}

\section{Generalization of the modulation spaces}\label{MS}
We introduce and study in this section a generalization of the classical modulation spaces. Throughout the rest of the article, we only work with the symmetric case of Gelfand-Shilov spaces, that is, we assume that $\ast=\dagger$, and we retain the notation $M_p$ for the common sequence $M_p=A_p$, which satisfies the same assumptions as before. Furthermore, to ease the notation, we  write $\SSS^*(\RR^d)=\SSS^*_*(\RR^d)$ (so that $\SSS^{\emptyset}(\mathbb{R})$ is the Schwartz space $\SSS(\mathbb{R}^{d})$). Instead of calling a TMIB of class $*-*$, we will simply say that it is of class $*$.
\subsection{Short-time Fourier transform}
\noindent As a preparation, we discuss in this subsection some properties of the short-time Fourier transform (STFT) on ultradistribution spaces. The STFT of an ultradistribution $f\in\mathcal{S}'^{*}(\mathbb{R}^d)$ with respect to the window function $g\in\mathcal{S}^{*}(\mathbb{R}^d)$ is
$$V_g f(x,\xi)=\langle f, \overline{M_{\xi} T_{x}g}\rangle,
\quad x,\xi\in\mathbb{R}^d.$$
It is clear that the function $V_gf (x,\xi)$ is smooth in the variables $(x,\xi)$. Moreover, it is well-known \cite{C-P-R-T,GZ1} (see also \cite[Subsection 2.3]{d-v-ind2018} for very short proofs of these facts) that the mapping
$
V_g: \mathcal{S}^{*} (\mathbb{R}^d)\rightarrow \mathcal{S}^{*}(\mathbb{R}^{2d})
$
is continuous and it extends to a continuous mapping $V_g:\mathcal{S}'^{*} (\mathbb{R}^d)\rightarrow \mathcal{S}'^{*}(\mathbb{R}^{2d})$. The adjoint mapping $V^*_g:\mathcal{S}^{*}(\mathbb{R}^{2d})\rightarrow \mathcal{S}^{*}(\mathbb{R}^{d})$ is defined as
$$V^*_g\Phi(t)= \int_{\RR^{2d}} \Phi(x,\xi)g(t-x)e^{2\pi i \xi \cdot t}d\xi dx
$$
and it extends continuously to a mapping $V^*_g:\mathcal{S}'^{*}(\mathbb{R}^{2d})\rightarrow \mathcal{S}'^{*}(\mathbb{R}^{d})$. The latter extension can be accomplished via a duality definition: $\langle V^*_g G,\varphi\rangle=\langle G, \overline{V_g\bar{\varphi}}\rangle$, for $G\in \mathcal{S}'^{*}(\mathbb{R}^{2d})$ and $\varphi\in \mathcal{S}^{*}(\mathbb{R}^{d})$. We note that it is trivial to verify that
\[
V_{g_{1}}^*V_{g_{2}}=(g_{1},g_{2})_{L^2(\mathbb{R}^{d})}\:\mathrm{Id}
\]
holds on $\mathcal{S}'^{*}(\mathbb{R}^d)$.

Let $E$ be either a TMIB or a DTMIB of class $*$ on $\mathbb{R}^{d}$. In the rest of this subsection we are interested in the action of the STFT on $E$. It is convenient to rewrite the STFT in a ``free $x$ variable way" as a convolution, that is,
$$
V_{g}f(\cdot, \xi)= \overline{\check{g}} \ast (M_{-\xi}f), \quad \xi\in\mathbb{R}^{d}.
$$
From this expression we see that if $f\in E$ then we can allow the window $g$ to be an element of the Beurling algebra $L^{1}_{\check{\omega}_{E}}$. Proposition \ref{propdoublemodule} and Corollary \ref{TMIcordualmodule} immediately yield the following simple but useful result.

\begin{proposition}\label{ocenaV_g} Let $f\in E$ and $g\in L^1_{\check{\omega}_E}$. Then, $V_{g}f(\cdot,\xi)\in E$ and
\[
\|V_gf(\cdot,\xi)\|_E\leq \nu_{E}(\xi)\|f\|_E\|g\|_{L^1_{\check{\omega}_E}},\quad \forall \xi\in\RR^d.
\]
Furthermore, $V_{g}f(\cdot,\xi)$ can be represented as the Pettis integral
\beq\label{TMeq4.2}
V_gf(\cdot,\xi)=\int_{\RR^d} \overline{g(t)} T_{-t}M_{-\xi}fdt, \quad \xi\in\RR^d.
\eeq
When $E$ is a TMIB, then \eqref{TMeq4.2} is actually a Bochner integral.
\end{proposition}

\begin{remark}\label{kvsmrf157}
Assume $E$ is a TMIB. For $f$ in the Banach space $\bar{E}=\{u\in \SSS'^*(\RR^d)|\, \bar{u}\in E\}$ with norm $\|u\|_{\bar{E}}=\|\bar{u}\|_E$, one can naturally define $V_g f$ even when $g\in E'$ via $V_g f(x,\xi)=\overline{\langle \bar{f},M_{\xi}T_x g\rangle}$. Notice that $V_g f$ is continuous on $\RR^{2d}$ and
\[
|V_g f(x,\xi)|\leq \check{\omega}_{E}(x)\check{\nu}_{E}(\xi)\|\bar{f}\|_E\|g\|_{E'}.
\]
\end{remark}

\bigskip

We now consider the adjoint STFT of a function $G\in L^{1}_{\omega_{E}\otimes \check{\nu}_{E}}$, where $E$ is either a TMIB or a DTMIB of class $*$ on $\RR^{d}$. If $\Phi(x,\xi)$ is integrable with respect to $\xi$, denote its partial inverse Fourier transform with respect to the second variable as $\mathcal{F}^{-1}_{2}\Phi(x,z)=\int_{\mathbb{R}^{d}}\Phi(x,\xi) e^{2\pi i z\cdot \xi }d \xi$; the definition obviously extends to ultradistributions by duality.\\
\indent If $G\in L^1_{\omega_E\otimes\check{\nu}_E}$ then $\mathcal{F}^{-1}_2G(x,\cdot)\in A_{\nu_E}$ and for $g\in E$, we have
\beq\label{smvklr157}
\|(T_xg)(\mathcal{F}^{-1}_2G(x,\cdot))\|_E\leq \omega_E(x) \|g\|_E \int_{\mathbb{R}^{d}}|G(x,\xi)|\check{\nu}_{E}(\xi)d\xi.
\eeq
Assume $g\in\SSS^*(\RR^d)$ and $G\in\SSS^*(\RR^{2d})$. Then $x\mapsto T_x g$, $\RR^d\rightarrow E$, and $x\mapsto \mathcal{F}^{-1}_2G(x,\cdot)$, $\RR^d\rightarrow A_{\nu_E}$, are continuous and hence so is the mapping
\beq\label{ktvsrn157}
x\mapsto (T_xg)(\mathcal{F}^{-1}_2G(x,\cdot)),\quad \RR^d\rightarrow E.
\eeq
Now (\ref{smvklr157}) implies that the function (\ref{ktvsrn157}) is Bochner integrable and (cf. \cite[Lemma 3.7]{TIBU})
\begin{equation}\label{TMeq4.3}
V^{\ast}_{g}G= \int_{\mathbb{R}^{d}} (T_{x}g) (\mathcal{F}^{-1}_{2}G (x,\cdot))dx,
\end{equation}
as a Bochner integral. A density argument implies that (\ref{TMeq4.3}) remains valid as a Bochner integral for $G\in L^1_{\omega_E\otimes\check{\nu}_E}$; in fact, the function (\ref{ktvsrn157}) is a pointwise limit a.e. of continuous functions, hence strongly measurable, and (\ref{smvklr157}) gives its Bochner integrability. If $E$ is TMIB, then a similar density argument for $g$ yields that (\ref{TMeq4.3}) is also true even when $g\in E$ and $G\in L^1_{\omega_E\otimes\check{\nu}_E}$. Assume now $E=F'$ is a DTMIB with $F$ being a TMIB and $g\in E$. Because of Remark \ref{kvsmrf157}, we can define $V^*_g G\in \SSS'^*(\RR^d)$ by $\langle V^*_gG,\varphi\rangle=\langle G,\overline{V_g\bar{\varphi}}\rangle$, $\varphi\in\SSS^*(\RR^d)$, and $|\langle V^*_gG,\varphi\rangle|\leq \|G\|_{L^1_{\omega_E\otimes\check{\nu}_E}}\|g\|_E \|\varphi\|_F$, that is, $V^*_g G\in E$. We can find $g_n\in\SSS^*(\RR^d)$, $n\in\ZZ_+$, such that $g_n\rightarrow g$ weakly* in $E$ (see Corollary \ref{vsktjr157}). Notice that for each $f\in F$, $V_{g_n}\bar{f}\rightarrow V_g \bar{f}$ pointwise. Since $\{g_n\}$ is weakly bounded in $E$, the Banach-Steinhaus theorem implies that it is strongly bounded and dominated convergence yields $\langle V^*_{g_n}G,f\rangle\rightarrow \langle V^*_g G,f\rangle$. Hence we can evaluate (\ref{TMeq4.3}) at $f\in F$ with $g_n$ in place of $g$ and take the limit to conclude that (\ref{TMeq4.3}) remains true as a Pettis integral even when $g\in E$. We have therefore shown:

\begin{proposition}\label{V^*_g}
Let $G\in L_{\omega_{E}\otimes\check{\nu}_{E}}^1$ and $g\in E$. Then, $V^{\ast}_{g}G\in E$,
\[
\|V_g^*G\|_E\leq \|g\|_E\|G\|_{L^1_{\omega_{E}\otimes\check{\nu}_{E}}},
\]
and \eqref{TMeq4.3} holds as an $E$-valued Pettis integral. When $E$ is a TMBI, then \eqref{TMeq4.3} is actually a Bochner integral.
\end{proposition}

\subsection{New modulation spaces}
Let $g\in\mathcal{S}^*(\mathbb{R}^d)\backslash\{0\}$. Let $F$ be either a TMIB or a DTMIB of class $*$ on $\mathbb{R}^{2d}$. We now define a new modulation space associated to $F$ as
\begin{equation*}
\mathcal{M}^{F}=\{f\in\mathcal{S}'^{*}(\mathbb{R}^{d})|\, V_gf\in F\}
\end{equation*}
provided with the norm $\|f\|_{\mathcal{M}^F}=\|V_gf\|_F$. Note first that
we always have $\mathcal S^*(\mathbb R^d)\subseteq \mathcal{M}^F$, so that this modulation space is non-trivial. We will show in Corollary \ref{ktrspv157} below that the definition of $\mathcal{M}^{F}$ does not depend on the chosen window $g$.

The spaces $\mathcal{M}^F$  are
generalizations of the well-known spaces
$\mathcal{M}^{L^{p,q}_\eta}=M^{p,q}_{\eta}(\mathbb{R}^{d})$  considered by many authors (cf.
\cite{F, fg11,fg22,gr, PilT,Toft2012}). Theorem \ref{EE'} and Remarks \ref{TMrk1}--\ref{TMrk5} provide examples of spaces $F$ that are TMIB and DTMIB, many of them differing from $L^{p,q}_\eta$.

\begin{proposition}\label{V^*}
Let $g_1,g_2\in \mathcal{S}^*(\mathbb{R}^d)$. Then
\begin{itemize}
\item [(i)]
\begin{equation}\label{ineq}
\|V_{g_1}V^*_{g_2}G\|_F\leq \|G\|_F\int_{\RR^{2d}}|V_{g_1}g_2(x,\xi)|\omega_F(x,\xi)\nu_F(\xi,0)dxd\xi.
\end{equation}
\item [(ii)]
The mappings
\begin{equation}\label{cont}
V_g^*:F\rightarrow \mathcal{M}^{F}\mbox{\,\,and\,\,}V_g:\mathcal{M}^{F}\rightarrow F,
\end{equation}
are continuous.
\end{itemize}
\end{proposition}

\begin{proof}
First we prove (i). Let $G\in \mathcal{S}^*(\mathbb{R}^{2d})$. A quick computation gives
\begin{equation*}
V_{g_1}V^*_{g_2} G(u,v)=
\int_{\mathbb{R}^{2d}}e^{-2\pi i(v-\xi)\cdot x}G(x,\xi)V_{g_1}g_2(u-x,v-\xi)dxd\xi,
\end{equation*}
or which amounts to the same,
\begin{align*}
V_{g_1}V^*_{g_2} G(u,v)
=\int_{\mathbb{R}^{2d}}(T_{(x,\xi)}M_{(-\xi,0)}G)(u,v)V_{g_1}g_2(x,\xi)dxd\xi.
\end{align*}
If $F$ is a TMIB of class $*$, using \cite[Lemma 3.7, p. 151]{TIBU} and the fact that $\mathcal{S}^{\ast}(\mathbb{R}^{2d})$ is dense in $F$, we obtain $V_{g_1}V^{\ast}_{g_2} G\in F$ and the  $F$-valued Bochner integral representation
\begin{equation}
\label{TMeq4.9}
V_{g_1}V^*_{g_2} G
=\int_{\mathbb{R}^{2d}}V_{g_1}g_2(x,\xi)T_{(x,\xi)}M_{(-\xi,0)}Gdxd\xi,
\end{equation}
which is now valid for all $G\in F$. Dually, $V_{g_1}V^{\ast}_{g_2} G\in F$ and \eqref{TMeq4.9} holds as a Pettis integral for all $G\in F$ if $F$ is a DTMIB of class  $*$. The estimate \eqref{ineq} is a direct consequence of \eqref{TMeq4.9}.

The continuity of the first mapping from part (ii) follows from (i) by choosing $g_1=g_2=g$, while the second mapping is continuous just by definition of $\mathcal{M}^{F}$.
\end{proof}

\begin{corollary}\label{ktrspv157}
The space $\mathcal{M}^{F}$ is independent of the window $g\in\mathcal{S}^*(\mathbb{R}^d)\backslash\{0\}$. Different windows induce equivalent norms on
$\mathcal{M}^{F}$.
\end{corollary}

\begin{proof}
Let $g_1,g_2\in\mathcal{S}^*(\mathbb{R}^d)\backslash\{0\}$ be two different windows. Proposition \ref{V^*} implies
\begin{equation*}\label{eq1}
\|V_{g_1}f\|_{F}=\frac{\|V_{g_1}V_{g_2}^*V_{g_2}f\|_{F}}{\|g_2\|^{2}_{L^2}}\leq \frac{\|V_{g_2}f\|_{F}} {\|g_2\|^{2}_{L^2}} \int_{\RR^{2d}}|V_{g_1}g_2(x,\xi)|\omega_F(x,\xi){\nu}_F(\xi,0)dxd\xi.
\end{equation*}
Interchanging $g_1$ and $g_2$ yields the assertion.
\end{proof}

We also have,

\begin{corollary}\label{iso}
The space $\mathcal{M}^{F}$ is a Banach space with the norm $\|\cdot\|_{\mathcal{M}^{F}}=\|V_g(\cdot)\|_{ F}$. Furthermore, $V_g^*:F\rightarrow\mathcal{M}^{F}$ is a surjective topological homomorphism.
\end{corollary}

\begin{proof}
Let $f_k$ be Cauchy sequence in $\mathcal{M}^{F}$. Then, the sequence $V_gf_k$ is Cauchy in the Banach space $F$ and converges to some $G$ in $F$.
Applying the continuity of the mapping $V^*$ given in \eqref{cont}, we obtain that $V_g^*V_gf_k\rightarrow V_g^*G$ in $\mathcal{M}^{F}$.
Hence, $f_k\rightarrow V^*_gG/\|g\|^{2}_{L^2}$ in $\mathcal{M}^{F}$. The last part is trivial (because of the open mapping theorem).
\end{proof}
\begin{corollary}\label{reflexivecor}
If $F$ is reflexive, then $\mathcal{M}^{F}$ is reflexive.
\end{corollary}
\begin{proof} A quotient of a reflexive Banach space by a closed subspace is reflexive; the assertion is then a consequence of Corollary \ref{iso}.
\end{proof}

 We now obtain an important property of $\mathcal{M}^{F}$ when $F$ is a TMIB of class $\ast$. We need to introduce some notation. We define the Banach space $\check{F}_2$ consisting of all $f\in\SSS'^*(\RR^{2d})$ such that $\check{f}_2(x,\xi):=f(x,-\xi)\in F$ with norm $\|\check{f}_2\|_{\check{F}_2}=\|f\|_F$. Clearly, $\check{F}_2$ is again a TMIB of class $*$ on $\RR^{2d}$ and its dual is canonically isomorphic to $(F')\check{}_2$; we denote the latter space simply as $\check{F}_{2}'$. Note that Theorem 4.8 (iii) extends \cite[Theorem 4.9]{fg11} applied to the case of classical modulation spaces.

\begin{theorem} Let $F$ be a TMIB of class $\ast$ on $\mathbb{R}^{2d}$.
\begin{itemize}
\item [(i)] The space $\mathcal{M}^{F}$ is a TMIB of class $*$ on $\mathbb{R}^{d}$.
\item [(ii)] Regardless the particular choice of $g\in\SSS^*(\RR^d)\backslash\{0\}$ inducing the norm on  $\mathcal{M}^F$, we always have the inequalities
\beq\label{ktroiv113}
\omega_{\mathcal{M}^F}(x)\leq \omega_F(x,0)\nu_F(0,x) \quad \mbox{and} \quad  \nu_{\mathcal{M}^F}(x)\leq \omega_F(0,-x), \quad \forall x\in\RR^d.
\eeq
\item [(iii)] The strong dual of $\mathcal{M}^{F}$ is the Banach space $\mathcal{M}^{\check{F}_{2}'}$.
\end{itemize}
\end{theorem}

\begin{proof}
Clearly $\mathcal{S}^*(\mathbb{R}^d)\subseteq\mathcal{M}^{F}\subseteq \mathcal{S}'^{*}(\mathbb{R}^d)$. Since the diagram
\begin{center}
\begin{tikzpicture}
  \matrix (m) [matrix of math nodes,row sep={3em},column sep={3em}]
  {
   \SSS^*(\RR^{2d}) & F & \SSS'^*(\RR^{2d}) \\
   \SSS^*(\RR^d) & \mathcal{M}^F & \SSS'^*(\RR^d)\\};
   \draw[right hook->] (m-1-1) -- (m-1-2) node[midway,above]{$\mathrm{Id}$};
   \draw[right hook->] (m-1-2) -- (m-1-3) node[midway,above]{$\mathrm{Id}$};
   \draw[->] (m-2-1) -- (m-1-1) node[midway,left]{$V_g$};
   \draw[->] (m-2-1) -- (m-2-2) node[midway,above]{$\mathrm{Id}$};
   \draw[->] (m-2-2) -- (m-1-2) node[midway,left]{$V_g$};
   \draw[->] (m-2-2) -- (m-2-3) node[midway,above]{$\mathrm{Id}$};
   \draw[->] (m-2-3) -- (m-1-3) node[midway,left]{$V_g$};
\end{tikzpicture}
\end{center}
commutes and all vertical mappings are injective topological homomorphisms (topological isomorphisms onto their images), both inclusions on the bottom row are continuous. The density of these inclusions follows from the diagram, Proposition \ref{V^*} and the fact $V_g^*V_g=\|g\|^{2}_{L^2} \operatorname*{Id}$. So $\mathcal{M}^{F}$ satisfies condition (a) from Definition \ref{defmod}. For any $f\in\mathcal{S}'^{\ast}(\mathbb{R}^{d})$, we have
$$
V_{g}(T_x f)= M_{(0,-x)}T_{(x,0)} V_g f  \quad \mbox{and } \quad V_{g}(M_{\xi}f)= T_{(0,\xi)} V_{g}f,
$$
which yields condition (b) from Definition \ref{defmod} and the inequalities \eqref{ktroiv113} (so that (c) from Definition \ref{defmod} also holds). This shows (i) and (ii).

Let us now show (iii). From (i), we obtain $(\mathcal{M}^F)'\subseteq \SSS'^*(\RR^d)$. Let $g\in\SSS^*(\RR^d)$ be a real normalized window, that is, such that $\|g\|_{L^2}=1$; we employ the norm on $\mathcal{M}^F$ induced by $g$. Notice that, for $\varphi\in\SSS^*(\RR^d)$, we have $\overline{V_g\bar{\varphi}(x,\xi)}=V_g\varphi(x,-\xi)$, $x,\xi\in\RR^d$. Let $f\in \mathcal{M}^{\check{F}'_2}$. For arbitrary $\varphi\in \mathcal{S}^*(\mathbb{R}^d)$, we have
\begin{equation*}
\left|\langle f, \varphi\rangle\right|
=\left|\langle f,\overline{V^*_gV_g\bar{\varphi}}\rangle\right|
= \left|\langle V_g f,\overline{V_g\bar{\varphi}}\rangle\right|
 \leq
\|V_gf\|_{ \check{F}'_2}\|V_g\varphi(x,-\xi)\|_{\check{F}_2}= \|f\|_{\mathcal{M}^{\check{F}'_2}}\|\varphi\|_{\mathcal{M}^{F}}.
\end{equation*}
Hence, $f\in(\mathcal{M}^{F})'$. Clearly, this mapping from $\mathcal{M}^{\check{F}'_2}$ to $(\mathcal{M}^F)'$ is a continuous injection. Conversely, let $f\in\left(\mathcal{M}^{F}\right)'\subseteq \mathcal{S}'^{\ast}(\mathbb{R}^{d})$. Since $\mathcal{M}^{F}$ is isometrically isomorphic to the closed subspace $V_g(\mathcal{M}^F)$ of $F$ and $A=\{u\in \check{F}_2|\, u(x,-\xi)\in V_g(\mathcal{M}^F)\}$ is a closed subspace of $\check{F}_2$, we can define a continuous functional $S$ on $A$ by $\langle S,V_g\chi(x,-\xi)\rangle=\langle S,\overline{V_g\overline{\chi}}\rangle:=\langle f,\chi \rangle$, $\chi\in \mathcal{M}^F$. We extend this functional to $\check{F}_2$ via the Hahn-Banach theorem, and denote it by $\tilde{f}\in \check{F}'_2\subseteq\SSS'^*(\RR^{2d})$. For $\varphi\in\SSS^*(\RR^d)$, we have
\beqs
\langle f,\varphi \rangle=\langle \tilde{f},\overline{V_g\bar{\varphi}}\rangle=\langle V^*_g\tilde{f},\varphi\rangle.
\eeqs
Clearly, $V^*_g\tilde{f}\in \mathcal{M}^{\check{F}'_2}$ (see Proposition \ref{V^*} (ii)) and $f$ is given by $\varphi\mapsto \langle V^*_g\tilde{f},\varphi\rangle$, $\mathcal{M}^F\rightarrow \CC$. Hence, $f=V^*_g\tilde{f}$ which shows that the mapping $\mathcal{M}^{\check{F}'_2}\rightarrow (\mathcal{M}^F)'$ is bijective. Thus, (iii) now follows from the open mapping theorem.
\end{proof}

We obtain the following corollary.

\begin{corollary}\label{dualm} If $F$ is a DTMIB of class $\ast$ on $\mathbb{R}^{2d}$, then $\mathcal{M}^{F}$ is a DTMIB of class $\ast$ on $\mathbb{R}^{d}$.
\end{corollary}

We now apply Proposition \ref{V^*_g} to show the minimality and maximality of the modulation spaces of the form $\mathcal{M}^{L^{1}_{\omega\otimes\nu}}$ and $\mathcal{M}^{L^{\infty}_{(1/\omega)\otimes(1/\nu)}}$.

\begin{proposition}\label{ktrisv157}
Let $E$ be a TMIB of class $*$ on $\mathbb{R}^d$, and write $\omega=\omega_{E}$ and $\nu=\nu_{E}$. Then,
$$
\mathcal{M}^{L_{\omega\otimes\check{\nu}}^1}\hookrightarrow E  \quad \mbox{ and } \quad \mathcal{M}^{L_{\check{\omega}\otimes\check{\nu}}^1}\rightarrow E' \to \mathcal{M}^{L_{(1/\omega)\otimes(1/\nu)}^{\infty}}.
$$
In particular, if $E$ is reflexive
$$
\mathcal{M}^{L_{\omega\otimes\check{\nu}}^1}\hookrightarrow E \to \mathcal{M}^{L_{(1/\check{\omega})\otimes(1/\nu)}^{\infty}}.
$$
\end{proposition}

\begin{proof} The continuous inclusions relations $\mathcal{M}^{L_{\omega\otimes\check{\nu}}^1}\to E  $ and $\mathcal{M}^{L_{\check{\omega}\otimes\check{\nu}}^1}\rightarrow E'$ follow from Proposition \ref{V^*_g} and the inversion formula $V^*_gV_g = \|g\|^{2}_{L^{2}}\mathrm{Id}$. The density of $\mathcal{S}^*(\mathbb{R}^d)$ in $E$ implies that $\mathcal{M}^{L_{\omega\otimes\check{\nu}}^1}$ is dense in $E$. Let $f\in E'$, then (we choose a real valued $g$)
$$
|V_{g}f(x,\xi)|= |\langle M_{-\xi} f, T_{x}g\rangle|\leq \nu(\xi) \omega(x) \|f\|_{E'}\|g\|_{E},
$$
which yields the continuous inclusion $E' \to \mathcal{M}^{L_{(1/\omega)\otimes(1/\nu)}^{\infty}}$.
\end{proof}

It is worth pointing out that minimality results for certain translation invariant Banach spaces of (Radon) measures on locally compact Abelian groups were obtained by Feichtinger in \cite{F81} and \cite[Proposition 5]{F90} under non-quasianalyticity type assumptions. A method (based on Gabor expansions) to establish the minimality of  the modulation spaces of type $\mathcal{M}^{L^{1}_{\eta}(\mathbb{R}^{2d})}$ for a rather general class of translation-modulation invariant Banach spaces is discussed in \cite[pp. 251--254]{gr}, see particularly \cite[Theorem 12.1.9]{gr}.

We now specialize Proposition \ref{ktrisv157} to modulation spaces. Let $F$ be either a TMIB or a DTMIB of class $*$ on $\RR^{2d}$. We write $\tilde{\omega}_{F}$ and $\tilde{\nu}_{F}$ for the functions
$$
\tilde{\omega}_{F}(x)= \omega_{F}(x,0)\nu_{F}(0,x) \quad \mbox{and} \quad \tilde{\nu}_{F}(x)= \omega_F(0,-x), \quad x\in \mathbb{R}^{d}.
$$
Because of (\ref{ktroiv113}), $\tilde{\omega}_{F}\geq \omega_{\mathcal{M}^{F}}$ and
$\tilde{\nu}_{F}\geq \nu_{\mathcal{M}^{F}}$. Proposition \ref{ktrisv157} implies the following result.

\begin{corollary}
Let $F$ be a Banach space of distributions on $\mathbb{R}^{2d}$.
\begin{itemize}
\item [(i)] If $F$ is a TMIB of class $*$, then
$
\mathcal{M}^{L_{\tilde{\omega}_{F}\otimes\check{\tilde{\nu}}_{F}}^1} \hookrightarrow \mathcal{M}^{F}.
$
\item [(ii)] If $F$ is a DTMIB of class $*$, then
$
\mathcal{M}^{L_{\tilde{\omega}_{F}\otimes\check{\tilde{\nu}}_{F}}^1} \rightarrow \mathcal{M}^{F} \rightarrow \mathcal{M}^{L_{(1/\check{\tilde{\omega}}_{F})\otimes(1/\tilde{\nu}_{F})}^{\infty}}.
$
\item [(iii)] If $F$ is a reflexive TMIB of class $*$, then
$
\mathcal{M}^{L_{\tilde{\omega}_{F}\otimes\check{\tilde{\nu}}_{F}}^1} \hookrightarrow \mathcal{M}^{F} \rightarrow \mathcal{M}^{L_{(1/\check{\tilde{\omega}}_{F})\otimes(1/\tilde{\nu}_{F})}^{\infty}}.
$

\end{itemize}

\end{corollary}

\section{More on modulation spaces constructed by TMIB tensor product spaces}\label{ex}
Theorem \ref{EE'} gives a way to construct instances of TMIB as completed tensor products of classical $L^p(\RR^d)$ spaces that do not have to coincide with the $L^p$-space over $\RR^{2d}$. As we noted in Remark \ref{TMrk4}, $L^2(\RR^d)\hat{\otimes}_{\pi} L^2(\RR^d)$ is a dense subspace of $L^2(\RR^{2d})$ which does not coincide with any $L^{p}(\mathbb{R}^{2d})$. The following proposition states an interesting property of the spaces $\mathcal{M}^{L^{p_1}(\RR^d)\hat{\otimes}_{\pi} L^{p_2}(\RR^d)}$, $1\leq p_2\leq p_1\leq 2$: They are all contained in $L^1(\RR^d)$.

\begin{proposition}\label{pex}
Let $1\leq p_2\leq p_1\leq 2$.
\begin{itemize}
\item[(i)] $V^*_g$ maps $L^{p_1}(\RR^d)\hat{\otimes}_{\pi} L^{p_2}(\RR^d)$ continuously into $L^1(\RR^d)$.
\item[(ii)] $\mathcal{M}^{L^{p_1}(\RR^d)\hat{\otimes}_{\pi} L^{p_2}(\RR^d)}$ is continuously included into $L^1(\RR^d)$.
\end{itemize}
\end{proposition}

\begin{proof} Notice that (ii) directly follows from (i). To prove (i) recall that every $f\in L^{p_1}(\RR^d)\hat{\otimes}_{\pi}L^{p_2}(\RR^d)$ is the sum of an absolutely convergent series
\[
f=\sum_{j=0}^\infty \lambda_j\varphi_j\otimes\psi_j,
\]
where $\lambda_j\in\CC$, $j\in\NN$, is such that $\sum_{j=0}^\infty |\lambda_j|<\infty$ and $\varphi_j,\psi_j\in\SSS(\RR^d)$, $j\in\NN$, are such that $\varphi_j\rightarrow0$ and $\psi_j\rightarrow0$ in $L^{p_1}(\RR^d)$ and $L^{p_2}(\RR^d)$, respectively (see \cite[Theorem 6.4, p. 94]{Sch}; $\SSS(\RR^d)$ is dense in $L^{p_k}(\RR^d)$, $k=1,2$). Hence, the series converges absolutely in $\SSS'(\RR^{2d})$ as well and we have $V^*_gf=\sum_j \lambda_jV^*_g(\varphi_j\otimes \psi_j)$. Notice that
\begin{equation}
\label{formulaV*tensors}
V^*_g(\varphi_j\otimes\psi_j)(t)=(\mathcal{F}^{-1}\psi_j)(t) \cdot (\varphi_j*g)(t).
\end{equation}
Let $q_1$ be the conjugate index of $p_1$ (i.e., $p_1^{-1}+q_1^{-1}=1$).  Let  $r =p_1p_2/(p_1p_2+p_2-p_1)$, that is $1+1/p_1=1/r+1/p_2$. Clearly $r\geq 1$ and the Hausdorff-Young and Young inequalities imply
\beqs
\|V^*_g(\varphi_j\otimes\psi_j)\|_{L^1}\leq \|\mathcal{F}^{-1}\psi_j\|_{L^{q_1}}\|\varphi_j*g\|_{L^{p_1}}\leq \|\psi_j\|_{L^{p_1}}\|\varphi_j\|_{L^{p_2}}\|g\|_{L^r}.
\eeqs
We conclude that $\sum_j \lambda_jV^*_g(\varphi_j\otimes \psi_j)$ converges absolutely in $L^1(\RR^d)$ and thus $V^*_gf\in L^1(\RR^d)$, namely, the image of $L^{p_1}(\RR^d)\hat{\otimes}_{\pi} L^{p_2}(\RR^d)$ under $V^*_g$ is contained in $L^1(\RR^d)$. Since the mapping $V^*_g:L^{p_1}(\RR^d)\hat{\otimes}_{\pi} L^{p_2}(\RR^d)\rightarrow \SSS'(\RR^d)$ is continuous,
the closed graph theorem implies that $V^*_g$ is a continuous mapping from $L^{p_1}(\RR^d)\hat{\otimes}_{\pi} L^{p_2}(\RR^d)$ into $L^1(\RR^d)$ and the proof of (i) is complete.
\end{proof}

As consequences of this proposition, we have the following results concerning the relationship between $\mathcal{M}^{L^{p_1}(\RR^d)\hat{\otimes}_{\pi} L^{p_2}(\RR^d)}$ and the classical modulation spaces $\mathcal{M}^{L^{p_1,p_2}(\RR^{2d})}$. We note that $\mathcal{M}^{L^{p}(\mathbb{R}^{2d})}=M^p(\mathbb{R}^{d})$ coincides with the space $W(\mathcal{F}L^{p}, l^{p})$ from \cite{F90}.

\begin{corollary}
Let $1\leq p_2\leq p_1\leq 2$ and $1< p\leq \infty$ be arbitrary. The spaces $\mathcal{M}^{L^{p_1}(\RR^d)\hat{\otimes}_{\pi}L^{p_2}(\RR^d)}$ and $\mathcal{M}^{L^{p}(\RR^{2d})}$ are different as sets.
\end{corollary}

\begin{proof} We use a window $g\in\SSS(\RR^d)\backslash \{0\}$.
By Proposition \ref{pex}, it suffices to
find $f\not\in L^1(\RR^d)$ such that
$V_g f\in L^p(\RR^{2d})$ for any $1<p\leq \infty$. For this purpose let
$f_1\in C^{\infty}(\RR)$ be nonnegative with support in $[1,\infty)$
and equal to $1/t$ on $[2,\infty)$. When $d=1$ we define $f$ to be just $f_1$
and when $d\geq 2$ let $\theta\in \mathcal{D}(\RR^{d-1})$ be nonnegative and define
$f=f_1\otimes \theta$. Clearly, $f\in C^{\infty}(\RR^d)\backslash L^1(\RR^d)$.
Since
$(\mathrm{Id}-\Delta_t)^de^{-2\pi i \xi t}=(1+4\pi^2|\xi|^2)^de^{-2\pi i\xi t}$,
we have
\begin{align*}
|V_g f(x,\xi)|&\leq (1+4\pi^2|\xi|^2)^{-d} \sum_{|\beta+\gamma|\leq 2d}
c_{\beta,\gamma} \int_{\RR^d} |D^{\gamma}f(t)| |D^{\beta}g(t-x)| dt\\
&=
(1+4\pi^2|\xi|^2)^{-d} \sum_{|\beta+\gamma|\leq 2d}
c_{\beta,\gamma} |D^{\gamma}f|*|D^{\beta}\check{g}|(x).
\end{align*}
Let $p_3,p_4>1$ be such that $p_3^{-1}+p_4^{-1}=1+p^{-1}$ (recall $p>1$).
By construction, $D^{\alpha} f\in L^{p_3}(\RR^d)$, $\forall \alpha\in\NN^d$.
Hence
\beqs
\|V_g f\|_{L^p(\RR^{2d})}\leq C\sum_{|\beta+\gamma|\leq 2d}
\| |D^{\gamma}f|*|D^{\beta}\check{g}|(\cdot)\|_{L^p(\RR^d)}\leq
C\sum_{|\beta+\gamma|\leq 2d}\|D^{\gamma}f\|_{L^{p_3}(\RR^d)}
\|D^{\beta}g\|_{L^{p_4}(\RR^d)},
\eeqs
i.e. $V_gf\in L^p(\RR^{2d})$.
\end{proof}

Since $L^2(\RR^d)\hat{\otimes}_{\pi} L^2(\RR^d)$ is continuously included into $L^2(\RR^{2d})$, the modulation space $\mathcal{M}^{L^2(\RR^d)\hat{\otimes}_{\pi} L^2(\RR^d)}$ is continuously and densely included into $L^2(\RR^{d})$ $(=\mathcal{M}^{L^2(\RR^{2d})}$). We therefore have,

\begin{corollary}\label{vtskle157}
The space $\mathcal{M}^{L^2(\RR^d)\hat{\otimes}_{\pi} L^2(\RR^d)}$ is a proper subspace of $L^{2}(\mathbb{R}^{d})$ with strictly finer topology.
\end{corollary}

\bigskip

\subsection*{Acknowledgement}
We thank Hans Feichtinger for useful comments and bringing our attention to many important bibliographical references.

\end{document}